\documentclass[a4paper,english]{article}

\usepackage[utf8]{inputenc}%caractères accentués
\usepackage[T1]{fontenc}
\usepackage{babel}
\usepackage[babel]{csquotes}

\usepackage{amsmath}%substack, displaystyle
\usepackage{amsfonts}%mathbb, mathfrak, mathcal
\usepackage{amsthm}%newtheorem, proof
\usepackage{bbm}%fonction indicatrice
\usepackage{amssymb}%sphericalangle
\usepackage{IEEEtrantools}%equations trop longues
\usepackage{url}%tilde dans les url

\newtheorem{theorem}{Theorem}[section]

\newtheorem{lemma}[theorem]{Lemma}
\newtheorem{corollary}[theorem]{Corollary}
\theoremstyle{definition}
\newtheorem{definition}[theorem]{Definition}
\newtheorem{example}[theorem]{Example}
\newtheorem{remark}[theorem]{Remark}

\DeclareMathOperator{\rk}{rk}

\DeclareMathOperator{\Supp}{Supp}

\DeclareMathOperator{\diag}{diag}
\DeclareMathOperator{\diam}{diam}

\newcommand{\bcs}{\backslash}

\newcommand{\eps}{\varepsilon}

\newcommand{\SO}{\mathrm{SO}}
\newcommand{\BA}{\mathrm{BA}}

\newcommand{\R}{\mathbb{R}}
\newcommand{\Z}{\mathbb{Z}}
\newcommand{\Q}{\mathbb{Q}}
\newcommand{\N}{\mathbb{N}}
\newcommand{\NS}{\mathbb{S}}
\newcommand{\PP}{\mathbb{P}}

\newcommand{\bv}{\mathbf{v}}
\newcommand{\bw}{\mathbf{w}}
\newcommand{\be}{\mathbf{e}}
\newcommand{\x}{\mathbf{x}}

\newcommand{\y}{\mathbf{y}}

\newcommand{\bp}{\mathbf{p}}

\newcommand{\cC}{\mathcal{C}}
\newcommand{\cD}{\mathcal{D}}

\newcommand{\da}{Diophantine approximation}

\numberwithin{equation}{section}	
\newcommand\eq[2]{\begin{equation}\label{eq:#1}{#2}\end{equation}}
\newcommand {\equ}[1]     {\eqref{eq:#1}}

\newcommand\hd{Hausdorff dimension}

\newcommand\dist{\operatorname{dist}}

\title{Rational approximation on quadrics:\\ a simplex lemma and its consequences}
\author{Dmitry Kleinbock%
\thanks{D.K. was supported in part by NSF grant DMS-1600814.}
\ and Nicolas de Saxcé}
\date{}

\newcommand{\Addresses}{
  \bigskip
\begin{flushleft}
\footnotesize{

Dmitry Kleinbock,\par\nopagebreak
\textsc{Department~of~Mathematics,~Brandeis~University,~Waltham~MA~02454,~USA}\par\nopagebreak
\textit{E-mail address:} \texttt{kleinboc@brandeis.edu}

  \medskip

Nicolas~de~Saxcé,\par\nopagebreak
\textsc{CNRS, Université Paris 13 -- LAGA, 93430 Villetaneuse, FRANCE}\par\nopagebreak
\textit{E-mail address:} \texttt{desaxce@math.univ-paris13.fr}
}
\end{flushleft}
}

\begin{document}

\maketitle

\begin{abstract} 
We give elementary proof of stronger versions of several recent results  on intrinsic Diophantine approximation on rational quadric hypersurfaces $X\subset \PP^n(\R)$.  The main tool is a refinement of the simplex lemma, which essentially says that rational points  on $X$ which are sufficiently close to each other must lie on a totally isotropic rational subspace of $X$.
%  we give an upper bound for the uniform Diophantine exponent for simultaneous approximation to real numbers 
 %$\xi_1,\dots,\xi_d, \xi_1^2+\dots+\xi_d^2$, as well as to the numbers 
% $\xi_1,\dots,\xi_d$ under {the} condition $\xi^2_1+\dots+\xi_d^2 = 1$.
  \end{abstract}

\section{Introduction}

The classical theory of Diophantine approximation  studies the way points $\mathbf{x}\in\R^n$ are approximated by rational points $\frac{\mathbf{p}}{q}\in\Q^n$, taking into account the trade-off between the  size of $q$   and the distance between $\frac{\mathbf{p}}{q}$ and $\mathbf{x}$; see \cite{cassels, schmidt} for a general introduction. Sometimes  $\mathbf{x}$ is assumed to lie on a certain subset of $\R^n$, for example a smooth manifold $X$; this leads to the theory of \emph{Diophantine approximation on manifolds\/}, in which there is no distinction between rational points which do or do not lie in $X$ (this is referred to as {\it ambient\/} approximation).
%This subject has experienced rapid progress during the last two decades, see e.g.\ \cite{kleinbockmargulis} where a technique using dynamics on the space of lattices was first introduced, or \cite{beresnevich_groshev, beresnevich_ba} for more recent developments.

Let now $X$ be a rational quadric hypersurface of $\R^n$, let $\mathbf{x}\in X$ and let $\frac{\mathbf{p}}{q}\in\Q^n$ be such that the distance between $\mathbf{x}$ and $\frac{\mathbf{p}}{q}$ is less than $\psi(q)$, where $\psi$ is decaying fast enough, namely $\lim_{x\to\infty}x^2 \psi(x) = 0$. Then $\frac{\mathbf{p}}{q}$ must lie on $X$ whenever $q$ is large enough! This elementary observation, due to Dickinson and Dodson \cite{dickinsondodson} for $n = 2$ and more generally to Dru\c tu, see  \cite[Lemma 4.1.1]{drutu}, has in part motivated a new field of \emph{intrinsic approximation}, which examines the quality to which points on a manifold are approximated by rational points lying on that same manifold. 
The paper \cite{kleinbockmerrill}
%to which the present work can be considered a sequel and reading companion,
studies the case $X = \NS^{n-1}$, the unit sphere in $\R^{n}$.
%Its main results are  summarized in the next subsection.
Later in \cite{fkmsquadric} the results of \cite{kleinbockmerrill} were significantly strengthened and extended to the case of $X$ being an arbitrary rational quadric hypersurface. 
An even more general framework was developed in   \cite{fkmsgeneral}. Roughly speaking, in order to exhibit points on submanifolds  $X\subset \R^n$ which are close enough to rational points of $X$, one has to make use of  the structure of $X$ (indeed, in general it is not even guaranteed that $X\cap \Q^n$ is not empty).
On the other hand, it is shown in  \cite{fkmsgeneral} that to prove some  {\it negative} results, that is, to show that many points of   $X$   are not too close to rational points, one often does not need to know much about $X$.
The main tool on which the argument of  \cite{fkmsgeneral} is based is the {\it Simplex Lemma} originating in Davenport's work \cite{davenport}.
The version presented in \cite[Lemma~4.1]{fkmsgeneral} is very general -- it applies to any manifold embedded in $\R^n$ -- and at the same time %it is 
precise enough to yield some satisfying theorems in the case of quadric hypersurfaces.

The purpose of this note is to show that in the special case where $X$ is a rational quadric hypersurface, one can give more elementary and more geometric proofs of the results of \cite{fkmsgeneral}.
This new approach will also yield more precise theorems.
The main point is that one can prove a version of the simplex lemma with arbitrary hyperplanes replaced by $\Q$-isotropic subspaces of $X$;
this, in turn, yields refined information on the diophantine properties of $X$.

%\bigskip

A detailed account of the results that are derived here is given in the next section.
After that  in \S\ref{sec:simplex} we prove  the simplex lemma for quadrics, Lemma~\ref{simplexquadric}, which is central in all the subsequent developments.
Applications of the simplex lemma to Diophantine approximation on quadrics are presented in \S\ref{sec:diophantine}.
Those results are proved along the same lines as the analogous statements for Diophantine approximation in the Euclidean space $\R^n$, but the proofs are included to make the paper self-contained.
Finally, in \S\ref{sec:open} we discuss some open problems and possible further directions for the study of intrinsic Diophantine approximation on projective varieties.

\bigskip

 \noindent{\bf Acknowledgements.}  The authors are grateful to Emmanuel Breuillard, Nikolay Moshchevitin and Barak Weiss for motivating discussions and suggestions.

\section{General setting and main results of the paper}\label{exposition}

Since it will make the proofs more transparent, we shall from now on always work in the projective setting.
We denote by $\PP^{n}(\R)$ the $n$-dimensional real projective space.
The natural map from $\R^{n+1}$ to $\PP^{n}(\R)$ will be denoted by $\x\mapsto  [\x]$.
For $x,y$ in $\PP^{n}(\R)$ the distance between $x$ and $y$ is given by
\[ \dist(x,y):=\frac{\|\bv_x\wedge \bv_y\|}{\|\bv_x\|\|\bv_y\|},\]
where $\bv_x$ and $\bv_y$ are any nonzero vectors on $x$ and $y$, respectively, and $\|\cdot\|$ is the Euclidean norm.
If $v = [\bv]\in\PP^{n}(\Q)$, where $\bv = (v_1,\dots,v_{n+1})$ is an integer vector with coprime coordinates, 
%given by coprime integer homogeneous coordinates $[v_1:\dots:v_d]$, 
the {\sl height} of $v$ is simply 
\[ H(v):=\max_{1\leq i\leq n+1} |v_i|.\]
Given a point $x$ in $\PP^{n}(\R)$ we want to study how well $x$ is approximated by points $v$ in $\PP^{n}(\Q)$.

\begin{remark}\label{affproj}
In order to go back to the setting of %\da\ in 
$\R^n$, one can  consider an affine chart from an open subset of $\PP^{n}(\R)$ to $\R^{n+1}$.
For example, if $$U=\{[(x_1,\dots,x_{n+1})]\ :\ x_{n+1}\neq 0\},$$ one can use the chart
\[ \begin{array}{ccc}
U & {\overset{\Phi}{\longrightarrow}} & \R^{n},\\
{[(x_1,\dots,x_{n+1})]} & {\overset{\Phi}{\longmapsto}} & (\frac{x_1}{x_{n+1}},\dots,\frac{x_{n}}{x_{n+1}}).
\end{array} \] 
\end{remark}

We consider a projective rational quadric $X$, given as the set of zeros of a rational quadratic form $Q$ in ${n+1}$ variables.
Namely, for such $Q$ let us %denote by 
%\[
%L_Q := \{\x\in\R^{n+1}: Q(\x) = 0\}
%\]
%the {\sl light cone} of $Q$, and let 
consider \eq{X}{ X = [Q^{-1}(0)] = \big\{x\in \PP^n(\R) : x = [\x]\text{ with }Q(\x) = 0\big\}.}
%\[ X=\{x\in\PP^{n}\ |\ Q(x)=0\},\]
%\comm{A question: our $X$ has dimension $d-2$. So I am thinking: maybe it is better to shift the indices and start with $\R^{d+1}$, as was the case in  \cite{fkmsquadric}? This will probably be more convenient for the reader's understanding...}
%\comm{Maybe we don't need the notation $L_Q$ after all.}

%\smallskip
Let us say that a subspace $E\subset\R^{n+1}$ is  {\sl totally isotropic} {(with respect to $Q$)}
if $Q|_E \equiv 0$. If $E$ is as above, the projection $[E]\subset X$ of $E$ onto $\PP^{n}(\R)$ will be referred to as a totally isotropic projective subspace. Recall that the {\sl $\Q$-rank} $\rk_{\Q}X$ of the quadric $X$ is the maximal dimension of a %totally 
totally isotropic rational subspace of $\R^{n+1}$. If $\rk_{\Q}X > 0$, this is the same as the maximal dimension of a
totally isotropic rational projective subspace of $X$ plus one. 
%Note that rational totally isotropic one-dimensional subspaces of $\R^{n+1}$ are in one-to-one correspondence with the rational points 
%$X(\Q)$ of $X$. 
In particular, $\rk_{\Q}X > 0$ if and only if $X(\Q)\ne\varnothing$.

{\begin{example} \label{Sn} As an example of a transition between affine and projective quadrics, let us consider the case $X = \NS^{n-1}$. The affine sphere
\eq{sn}{Y = \big\{(y_1,\dots,y_n) : y_1^2 + \dots + y_n^2 = 1 \big\}}
can be identified with 
$$X =  \big\{ [(x_1,\dots,x_{n+1})] : x_1^2 + \dots + x_n^2 = x_{n+1}^2\big\}$$
via the map $\Phi$ described in Remark \ref{affproj}. Here $Q(\x) = x_1^2 + \dots + x_n^2 - x_{n+1}^2$, and the $\Q$-rank of $X$ is equal to $1$, with rational points of $X$ (or of $Y$) being in one-to-one correspondence with rational lines in $\R^{n+1}$ which are  totally isotropic with respect to $Q$.\end{example}}

Given a point $x$ in $X%(\R)
$, we shall be interested in the quality of rational approximations $v\in X(\Q)$ to %the point 
$x$. 
{Informally speaking, $v$ constitutes a reasonably good approximation to $x$ if $\dist(x,v)$ is small  and the height of $v$ is not too large. Note that if $v = [(\bp,q)]$, where  $(\bp,q)$ is a vector in $\Z^{n+1}$ with coprime coordinates, and $\Phi$ is the map defined in Remark \ref{affproj}, then $\Phi(v) = \bp/q$. When $x$ is close to $v$, $\y = \Phi(x)$ is close to $\bp/q$, therefore  the ratio of $H(v)$ and $|q|$ is bounded between two constants depending only on $x$. The same can be said about  the ratio of $\dist(x,v)$ and $\|\y - \bp/q\|$. Thus the quality of  rational approximations $v\in X(\Q)$ to %the point 
$x\in X$ is up to a multiplicative  constant the same as the quality of  rational approximations $\bp/q\in Y = \Phi(X)$ to $\y = \Phi(x)$.} 

 The basic theory of such approximations has been developed in \cite{fkmsquadric} by Fishman, Kleinbock, Merrill and Simmons. In particular it was proved there \cite[Theorem 5.1]{fkmsquadric} that if
%\footnote{\comm{I think we should define nonsingular, at least in a footnote.}} 
\eq{nonsing}
{\rk_{\Q}X > 0\quad\text{ and }\quad X\text{ is nonsingular
}}
%\comm{(should we explain  nonsingular?)}
(recall that a quadric hypersurface $X$ is said to be {\sl nonsingular} if the quadratic form that defines it is nondegenerate, i.e.\ has nonzero discriminant\footnote{This is also equivalent to $X$ being nonsingular as a projective algebraic variety.}),
then  for every $x\in X$ there exists $C_{x} > 0$ and a sequence $(v_k)_1^\infty$ in $X(\Q)$ such that
\begin{equation}
\label{dirichlet}
v_k\to  x \quad\text{ and }\quad \dist(v_k,x) \leq \frac{C_{x}}{H(v_k)}.
\end{equation}
{Previously this was established in \cite{kleinbockmerrill} by Kleinbock and Merrill   for $X = \NS^{n-1}$, in which case the constant $C_x$ can be chosen independently of $x$. In other words, in the affine language of Example \ref{Sn}, there exists a constant $C$ dependent only on $n$ such that for any $\y\in \NS^{n-1}$ there exist infinitely many $\bp/q\in  \NS^{n-1}$ with $\|\y - \bp/q\|\le C/|q|$.}

\smallskip

{Now for a point $x$ in a quadric hypersurface $X\subset \PP^{n}(\R)$, let us   define} the {(intrinsic)} {\sl Diophantine exponent} of $x$ by
\eq{beta}{ \beta(x) := \inf\left\{\beta>0\ |\ \exists\, c>0:\,\forall \,v\in X(\Q),\,
\dist(x,v)\geq cH(v)^{-\beta}\right\};}
then it follows that under the assumption \equ{nonsing}, $\beta(x) \geq 1$ for all $x\in X$.
On the other hand, it is shown in \cite[Theorem 1.5]{fkmsgeneral} that the opposite inequality $\beta(x) \leq 1$ is true for Lebesgue-almost every $x\in X$ 	in the generality when $X$ is not just a rational quadric but an arbitrary {non-degenerate} hypersurface.
Moreover, the same is true if the Lebesgue measure is replaced by an \emph{absolutely decaying} measure %from a certain class. 
(see \S\ref{sec:diophantine1}
%and  \cite{fkmsgeneral} 
for definitions and more detail).

This naturally leads to a question of exhibiting other measures $\mu$ on $X$ such that  $\beta(x) \leq 1$ for $\mu$-almost all $x\in X$.
This is reminiscent to the subject of \da\ on manifolds and fractals, which has been extensively developed during recent decades for ambient approximation in $\R^n$, see  \cite{bernikdodson, KLW, kleinbock-margulis}.
%\comm{Please add at least two more references:  Kleinbock, D. Y.; Margulis, G. A. Flows on homogeneous spaces and Diophantine approximation on manifolds. Ann. of Math. (2) 148 (1998), no. 1, 339–360. and Kleinbock, Dmitry; Lindenstrauss, Elon; Weiss, Barak On fractal measures and Diophantine approximation. Selecta Math. (N.S.) 10 (2004), no. 4, 479–523.} 
Measures satisfying the above property are usually called {\sl extremal}.
We shall also say that a submanifold $M\subset X$ is extremal if so is the Lebesgue measure on $M$ (by which we mean the restriction to $M$ of the $k$-dimensional Hausdorff measure where $k = \dim M$).

\smallskip

Our first theorem, which is actually a special case of a more general result, Theorem~\ref{extremalquadric}, refines \cite[Theorem 1.5]{fkmsgeneral} for rational quadrics $X$ as follows:

%Our first theorem  refines the aforementioned result of Fishman-Kleinbock-Merrill-Simmons \cite{fkmsgeneral}, is the following; the reader is referred to the beginning of Section~\ref{sec:diophantine} for the definition of isotropic decay of a measure.

\begin{theorem}[Extremality of %decaying measures
submanifolds of large dimension]
\label{exi}
Let $X$ be a rational quadric hypersurface in $\PP^{n}(\R)$, and let $M$ be a smooth submanifold of $X$ with $\dim M \geq\rk_{\Q}X$. Then $\beta(x) \le 1$ for Lebesgue-almost every $x\in M$.
%and assume that $X(\Q)\setminus\ker Q$ is non-empty.
%Let $\mu$ be an isotropically decaying measure on $X$.
%Then, for $\mu$-almost every $x\in X$,
%\[ \beta(x) = 1.\]
%This holds in particular if $\mu$ is the Lebesgue measure on a $k$-dimensional $C^1$ submanifold of $X$, with $k\geq\rk_{\Q}X$.
\end{theorem}

In the case where $X$ has $\Q$-rank one, the above theorem provides a very simple and satisfactory answer to the problem of Diophantine approximation on submanifolds of $X$: any positive-dimensional submanifold $M\subset X$ is extremal. Note that there is no {non-degeneracy} condition on the submanifold $M$.
This comes in contrast to the case of approximation in $\R^n$, where one has to require that the submanifold is not included in an affine subspace.
%When the quadric hypersurface has higher rank, the study of extremality of submanifolds is more subtle, and must require some quantitative non-divergence results in the space of lattices.
% \comm{(not sure if we need to state it here)}

\smallskip
In view of Theorem \ref{exi}, it is natural to ask, given a %sufficiently regular subset $K$ in $X$, 
submanifold $M$ of $X$ of dimension at least $\rk_{\Q}X$ and a fixed $\beta > 1$, how large the intersection $M\cap W_\beta$ can be, where $W_\beta$ denotes the set of points in $X$ whose Diophantine exponent is at least $\beta$.
Note that it was proved in \cite[Theorem 6.4]{fkmsquadric} that whenever $X$ satisfies \equ{nonsing}, the \hd\ of  $W_\beta$ is equal to $\frac{n-1}\beta$. Also in \cite{fishmanmerrillsimmons} some upper estimates for the \hd\ of $M\cap W_{\beta}$ was obtained in the case when $M$ supports an absolutely decaying and Ahlfors-regular  measure (see \S\ref{sec:diophantine2}
%and  \cite{fkmsgeneral} 
for details).
Our second application of the simplex lemma strengthens the main result of \cite{fishmanmerrillsimmons}. Here is a special case of a more general result, Theorem~\ref{hausdorffexponent}:

%In the case where $K$ is the support of an Ahlfors-regular measure, we derive below an upper bound for the Hausdorff dimension of $K\cap W_\beta$, depending on the regularity of $\mu$.
%We refer the reader to Theorem~\ref{hausdorffexponent} for the general statement, and only give here the result in the $\Q$-rank one case, which is easier to state.

\begin{theorem}[$\beta$-approximable points on submanifolds of large dimension]
\label{bei}
%Let $X$ be a rational quadric hypersurface of $\Q$-rank one,
%and $\mu$ an Ahlfors-regular measure of dimension $\delta$ on $X$.
%Writing $K=\Supp\mu$, we have, for every $\beta\geq 1$,
Let $X$ be a rational quadric hypersurface in $\PP^{n}(\R)$, and let $M$ be a $k$-dimensional smooth submanifold of $X$ with $k \geq\rk_{\Q}X$. Then one has 
%\comm{(Can you check that this is correct? or feel free to modify the statement)}
\[ \dim_H (M\cap W_\beta) \leq  k - (k+1 - \rk_{\Q}X)(1 - \tfrac{1}{\beta}).\]
%\frac{\delta}{\beta}.$$}
\end{theorem}

As the third application of our simplex lemma, we study the winning property of the set $\BA_X$ of badly approximable points on $X$.
In our setting,
\eq{defba}{ \BA_X := \{ x\in X \ |\ 
\exists\, c>0:\,\forall\, v\in X(\Q),\ \dist(x,v)\geq cH(v)^{-1}\}.}
We define a version of Schmidt's game using only totally isotropic rational subspaces and show the associated winning property for the set $\BA_X$ (Theorem \ref{winning}).
%, strengthening results from \cite{fkmsquadric, fkmsgeneral}}.
Here is a special case:

\begin{theorem}[Thickness of $\BA_X$ on submanifolds of large dimension]
\label{bai}
Let $X$ be a rational quadric hypersurface in $\PP^{n}(\R)$.
%, with $\Q$-rank $k\geq 1$.
Then for any $C^1$ submanifold $M\subset X$ of dimension at least $\rk_{\Q}X$,
\[ \dim_H(\BA_X\cap M) = \dim M.\]
\end{theorem}

The properties of the set $\BA_X$ have %already 
been studied in \cite{fkmsgeneral}.
In particular, it was shown \cite[Theorem~4.3]{fkmsgeneral} that $\BA_X$ is hyperplane absolute winning  (see \S\ref{sec:diophantine3}
%and  \cite{fkmsgeneral} 
for the definition and more detail); this gave the conclusion of the above theorem for $M = X$.
%manifolds of codimension at most one.
The refined version given above has the advantage that it is optimal: indeed, if $M$ is any totally isotropic rational projective subspace of $X$ of dimension $\rk_{\Q}X-1$, then $\BA_X\cap M=\varnothing$.

\smallskip

%The plan of the paper is as follows.
%In Section~3, we prove the simplex lemma for quadric hypersurfaces.
%In Section~4, we derive three applications to Diophantine approximation, which are more general versions of Theorems~\ref{exi}, \ref{bei} and \ref{bai}.

%\comm{Here we should say that in the next section we will state and prove the simplex lemma, and in \S\ref{sec:diophantine} we will formulate and prove more general versions of the preceding three theorems. I stopped here, if you agree with my changes, maybe you can continue changing the remaining part of the paper?}

\section{Diagonal flows and the simplex lemma}
\label{sec:simplex}

The purpose of this section is to derive a simplex lemma, Lemma~\ref{simplexquadric}, for rational points on a rational  quadric hypersurface $X\subset\R^{n+1}$. 
For the proof, we shall relate good rational approximations to $x\in X$ %in a quadric hypersurface 
to the behavior of some diagonal orbit in the space of lattices in $\R^{n+1}$.
%\comm{I think this sentence needs to be changed, not clear what `short vectors in some orbit' means... }
%But first, we define the relevant objects.
%\comm{Which ones? I thought everything has been defined already..}

Recall that the classical simplex lemma\footnote{Here we restate it using the projective language.} states that for each $n\in\N$ there exists $c=c(n)>0$ such that if $x$ is a point in %the projective space 
$\PP^{n}(\R)$ and $\rho\in(0,1)$, then there exists an $(n-1)$-dimensional projective subspace of $\PP^{n}(\R)$ containing all rational points with height at most $c\rho^{-\frac{n}{n+1}}$ inside the ball $B(x,\rho)$. 
For a proof of the simplex lemma, we refer the reader to \cite[Lemma~4]{kristensenthornvelani}.

Here we consider a rational quadratic form $Q$ on $\R^d$  and study rational points on %the associated hypersurface
%\[ X=\{x\in\PP^{n}\ |\ Q(x)=0\}.\]
%The distance and the height on $X$ will be those induced by the distance and the height on $\PP^{n}$, described above.
$X$ as in \equ{X}.

\begin{lemma}[Simplex lemma for quadric hypersurfaces]
\label{simplexquadric}
Let $X$ be a rational quadric hypersurface in $\PP^{n}(\R)$.
Then there exists $c>0$ such that, for every ball $B_\rho\subset X$ of radius $\rho\in(0,1)$, the set
\[ B_\rho \cap \{v\in X(\Q)\ |\ H(v)\leq c\rho^{-1}\} \]
is contained in a totally isotropic rational projective subspace %$L\leq X(\Q)$
of $X$.
\end{lemma}

Let $B_Q$ be the symmetric bilinear form associated to the quadratic form $Q$ defining $X$.
The {\sl kernel} of $Q$ is defined by
\[ \ker Q = \{ x=[\x]\in\PP^n(\R)\ |\ \forall \,\y\in\R^{n+1},\, B_Q(\x,\y)=0\}.\]
Assuming that $X(\Q)\smallsetminus \ker Q$ is non-empty,
%\comm{need to define $\ker Q$, obviously it is not the preimage of $0$,  and also need to mention that otherwise the statement is trivially satisfied}
% We comment on the case $X(\Q)\subset\ker Q$ at the beginning of the proof of Lemma~\ref{simplexquadric}.
we may write, in some rational basis of $\R^{n+1}$,
\begin{equation}\label{qgoodform}
Q(x_1,\dots,x_{n+1}) = 2x_1x_{n+1} + \tilde{Q}(x_2,\dots,x_{n}),
\end{equation}
where $\tilde{Q}$ is a quadratic form in $n-1$ variables.
Let $G=\SO_Q(\R)$ be the group of unimodular linear transformations of $\R^{n+1}$ preserving the quadratic form $Q$.
The group $G$ acts transitively on $X\smallsetminus \ker Q$, which may be identified with the quotient space $X\simeq P\bcs G$, where $P$ is the stabilizer of the isotropic line $[\be_1]$ in the standard representation.
In fact, for $x\in X\smallsetminus \ker Q$, we may choose $u_x\in G\cap O_{n+1}(\R)$ such that $u_xx=[\be_1]$.

We shall consider the diagonal subgroup %\comm{if you are saying `flow', need to explain on which space the flow is taking place} 
$a_t=\diag(e^{-t},1,\dots,1,e^t)$ in $G$, and if $x\in X$, let
%\marginpar{Again, $u_x$ should be chosen in a compact subset of $G$.}
\[ g_t^x = u_x^{-1} a_t u_x.\]
%Let $P=\{g\in G\ |\ a_tga_{-t}\ \mbox{is bounded for}\ t\geq 0\}$.
%We identify $\PP^{n}(\R)$ with $P\bcs G$ by the map
%\[\begin{array}{ccc}
%P\bcs G \to \PP^{n}(\R)\\
%g \mapsto g\cdot[\be_1].
%\end{array}.\]
The lemma below is due to Kleinbock and  Merrill \cite{kleinbockmerrill} in the case of projective spheres, and to Fishman, Kleinbock, Merrill and Simmons \cite[Lemma~7.1]{fkmsquadric} in the general case.

\begin{lemma}[Dani correspondence for quadric hypersurfaces]
\label{daniquadric}
Let $Q$ be %an integer quadratic form on $\R^d$, given by the equation 
as in \eqref{qgoodform}, and write $X$ for the associated rational quadric hypersurface in $\PP^{n}(\R)$.
With the above notation, there exists $C > 0$ such that for $x\in X$ and $v\in X%(\Q)
$, we have, for all $t\in\R$,
\[ \|g_t^x\bv\| \leq C\max(e^{-t}H(v), H(v)\dist(x,v), e^tH(v)\dist(x,v)^2),\]
where $\bv\in\Z^d$ is a representant of $v$ with coprime integer coordinates.
\end{lemma}
\begin{proof}
Since $\tilde{Q}$ is a quadratic form, we may choose $C_0\geq 2$ such that for all $\bw$ in $\R^{n-1}$, $|\tilde{Q}(\bw)|\leq C_0\|\bw\|^2$.
With $u_x$ as above, write
\[ u_x\bv= v_1\be_1+v_2\be_2+\dots+v_{n+1}\be_{n+1}.\]
Letting $\bw=v_2\be_2+\dots+v_{n}\be_{n}$, we have
\[ u_xg_t^x\bv= e^{-t}v_1\be_1+\bw+e^tv_{n+1}\be_{n+1},\]
and therefore, since $u_x$ is in $O_{n+1}(\R)$,
\begin{equation}\label{gtxv}
 \|g_t^x\bv\| \leq 3\max( e^{-t}|v_1|, \|\bw\|, e^t|v_{n+1}|).
\end{equation}
Of course, $|v_1|\leq H(v)$ and $\|\bw\|\leq\|\bw+v_{n+1}\be_{n+1}\|\leq H(v)\dist(x,v)$.
Moreover, $Q(u_x\bv)=0$ yields
$$|v_{n+1}| = \frac{|\tilde{Q}(\bw)|}{2|v_1|} \leq \frac{C_0\|\bw\|^2}{2|v_1|},$$
so that, provided $\dist(x,v)\leq\frac{\sqrt{2}}{2}$,
\[ |v_d| \leq \frac{C_0}{2}\frac{H(v)\dist(x,v)^2}{\sqrt{1-\frac{\dist(x,v)^2}{H(v)^2}}} \leq C_0H(v)\dist(x,v)^2.\]
Of course, if $\dist(x,v)\geq\frac{\sqrt{2}}{2}$, we also have $|v_d|\leq H(v)\leq C_0H(v)\dist(x,v)^2$, because $C_0\geq2$.
Going back to \eqref{gtxv}, we find the desired inequality, with $C=3C_0$.
\end{proof}

We can now prove the simplex lemma.

\begin{proof}[Proof of Lemma~\ref{simplexquadric}]
Let $Q$ be a quadratic form defining the hypersurface $X$.
The result is obvious if $X(\Q)\subset \ker Q$,  so we may assume that $X(\Q)\smallsetminus\ker Q$ is non-empty.
Then, replacing $Q$ if necessary by an integer multiple, we may find an integer basis of $\R^{n+1}$ in which $Q$ has the form \eqref{qgoodform}.

Fix a constant $C_1$ such that for all $\bv\in\R^{n+1}$, $|Q(\bv)|\leq C_1\|\bv\|^2$, and let $c=\frac{1}{C\sqrt{5C_1}}$, where $C$ is the constant given by Lemma~\ref{daniquadric}.
We need to show that any family $v_1,\dots,v_s$ of points in $X(\Q)\cap B(x,\rho)$ satisfying $H(v_i)\leq c\rho^{-1}$, $i=1,\dots,s$, generates a totally isotropic subspace. 
For each $v_i$, we take a representant $\bv_i$ in $\Z^{n+1}$ with coprime integer coordinates.
It is enough to show that for all $i$ and $j$, $Q(\bv_i\pm\bv_j)=0$, and since the quadratic form $Q$ takes integer values at integer points, it suffices to check that for all $i$ and $j$, $|Q(\bv_i\pm \bv_j)|$ is less than $1$.

Now, choosing $t>0$ such that $e^t=\rho^{-1}$, Lemma~\ref{daniquadric} shows that $\|g_t^x\bv_i\| \leq Cc$.
Then, we write
$$Q(\bv_i\pm \bv_j) = Q(g_t^x\bv_i\pm g_t^x\bv_j))
\leq C_1 \|g_t^x\bv_i\pm g_t^x\bv_j\|^2
\leq 4C_1 (C c)^2 =\frac{4}{5}.$$
This implies what we want.
\end{proof}

\begin{remark}\label{spheres}
In the case when $X=\NS^{n-1}$ is the $(n-1)$-dimensional sphere, %which may be 
identified with %the subset of $\R^{n}$ defined by the equation $x_1^2+\dots+x_{n}^2=1$, 
{$Y$ as in \equ{sn},} one can give a more direct proof of the simplex lemma.
Indeed, if $\frac{\bp_1}{q_1}$ and $\frac{\bp_2}{q_2}$ are two distinct rational points on $\NS^{n-1}$ of height at most $\frac{\rho^{-1}}{2}$, we have
\[
\left\|\frac{\bp_1}{q_1}-\frac{\bp_2}{q_2}\right\|^2
= 2 - \frac{(\bp_1,\bp_2)}{q_1q_2} \geq \frac{1}{q_1q_2} \geq 4\rho^2,\]
so that any open ball of radius $\rho$ contains at most one rational point of height at most $\frac{\rho^{-1}}{2}$.
In fact, such a direct computation can also be made for a general quadric hypersurface, but we chose to give a more geometric proof of Lemma~\ref{simplexquadric} here.
\end{remark}
%ALTERNATIVE PROOF FOR THE SIMPLEX LEMMA
%Take a=1/H(v_1) and b=-1/H(v_2). Identifying v_1, v_2 with their integer representants, one then has
%\[ Q(av_1+bv_2)   \leq C_Q. d(v_1,v_2)^2
%                         \leq 4.C_Q.max_{i=1,2} (d(v_i,x)^2)
%                         \leq 4.C_Q.max {i=1,2} (\rho/H(v_i))\]
%And since \rho\leq c\min(H(v_1)^{-1},H(v_2)^{-1}), we get
%\[ Q(av_1+bv_2) \leq 4.C_Q.c.H(v_1)^{-1}H(v_2)^{-1}.\]
%But since Q is an integer quadratic form and v_1, v_2 are isotropic, Q(av_1+bv_2) is a rational with denominator at most H(v_1)H(v_2), so it is at least H(v_1)^{-1}H(v_2)^{-1} unless it is zero.
%Taking c<(4C_Q)^{-1}, we find Q(av_1+bv_2)=0, so span(v_1,v_2) is totally isotropic, and the lemma is proved.

\begin{remark}
%In general, the dimension of a totally 
%isotropic linear subspace is bounded above by the $\Q$-rank of $\SO_Q$.
When the quadratic form $Q$ has $\Q$-rank one, %we get that 
the only isotropic rational projective subspaces are %the 
points in $X(\Q)$.
This %will 
makes the consequences of the simplex lemma more spectacular in the particular case of $\Q$-rank one.

%While we were finishing to write up this note, we became aware that the case where $Q$ has $\R$-rank one, i.e. where $X$ is a projective sphere, had already been studied by Beresnevich, Ghosh, Simmons and Velani \cite{bgsv} as part of the theory of diophantine approximation on limit sets of Kleinian groups.
%%This relates to the fact that diophantine approximation on a projective variety $P\bcs G$, with $G$ a semisimple $\Q$-group, is much easier when $G$ has rank one.
\end{remark}

\section{Applications to Diophantine approximation}
\label{sec:diophantine}

In this section, as before, $X$ is a rational quadric hypersurface in $\PP^{n}(\R)$ defined by a rational quadratic form $Q$.
We are concerned with  {intrinsic} Diophantine approximation on %the quadric hypersurface 
$X$, which is the study of the quality of approximations of a point $x$ in $X$ by rational points $v$  {lying on} $X$.
%\comm{this is probably kind of repetitive, we said it already in section 2..}
%Right, but it can only help the reader.
On that matter, the simplex lemma has several simple consequences, which we now explain.

\subsection{Extremality}\label{sec:diophantine1}

Recall that the Diophantine exponent of a point $x\in X$ was defined by \equ{beta}.
%\[ \beta(x) = \inf\{\beta>0\ |\ \exists c>0:\,\forall v\in X(\Q),\,
%d(x,v)\geq cH(v)^{-\beta}\}.\]
%In \cite{fkmsquadric}, it is proved that if $X(\Q)\smallsetminus\ker Q$ is non-empty, then almost every $x$ in $X$ satisfies $\beta(x)=1$.
%Following the terminology used in standard diophantine approximation in the projective space, we shall therefore say that a Borel measure $\mu$ on $X$ is \emph{extremal} if we have, for $\mu$-almost every $x$, $\beta(x)=1$.
%Using a very general version of the simplex lemma, the article \cite{fkmsgeneral} gives a sufficient decay condition for a measure to be extremal.
%That condition requires that the measure $\mu$ do not concentrate on hyperplanes of $\R^n$, see \cite[Theorem~1.5]{fkmsgeneral} for details.
Our next theorem generalizes Theorem \ref{exi} using the following definition.
%Using the version of the simplex lemma presented above, Lemma~\ref{simplexquadric}, we now show that this condition can be substantially weakened: it is enough to assume that $\mu$ does not concentrate on rational isotropic subspaces.
%The relevant definition is as follows.

\begin{definition}\label{def-ad}
Given a positive parameter $\alpha$, a finite Borel measure $\mu$ on the quadric hypersurface $X$ will be called {\sl $\alpha$-isotropically absolutely decaying}, abbreviated as {\sl $\alpha$-IAD}, if there exists a constant $C> 0$ such that for every $x\in X$ and every totally isotropic rational projective subspace $L\subset X$,
\eq{ad}{ \forall\,\eps>0 \ \forall\,\rho \in(0, 1),\quad 
\mu\big(B(x,\rho)\cap L^{(\eps\rho)}\big) \leq C\eps^\alpha\mu\big(B(x,\rho)\big),}
where $L^{(\tau)}$ denotes the neighborhood of size $\tau$ of the set $L$.
We shall say that $\mu$ is \emph{isotropically absolutely decaying} (IAD) if it is $\alpha$-IAD for some $\alpha>0$.
\end{definition}

\begin{theorem}[IAD measures are extremal]
\label{extremalquadric}
Let $X$ be a rational quadric in $\PP^{n}(\R)$, %and assume that $X(\Q)\setminus\ker Q$ is non-empty.
and let $\mu$ be an IAD measure on $X$.
Then $ \beta(x) \le 1$ for $\mu$-almost every $x\in X$.
%\[ \beta(x) \ge 1.\]
\end{theorem}

\begin{remark}
Recall that a measure $\mu$ is called {\sl $\alpha$-absolutely decaying} if \equ{ad} holds for some $C> 0$,  every $x\in X$ and \emph{every}  subspace $L\subset \PP^n(\R)$, and 
{\sl absolutely decaying} if it is $\alpha$-absolutely decaying for some $\alpha>0$. It follows from \cite[Theorem 1.5]{fkmsgeneral} that for any absolutely decaying measure $\mu$ on $X$ one has $\beta(x) \leq 1$ for $\mu$-almost every $x\in X$. In fact it holds more generally when $X$ is not just a rational quadric but an arbitrary {non-degenerate} smooth hypersurface.

Clearly absolutely decaying measures are IAD but not vice versa. In particular, it is clear that Lebesgue measure on a smooth proper submanifold $M$ of $X$ with $\dim M \geq\rk_{\Q}X$ is not absolutely decaying but $\alpha$-IAD with $\alpha = \dim M - \rk_{\Q}X + 1$; so Theorem \ref{exi} is a corollary from Theorem \ref{extremalquadric}.
\end{remark}

\begin{proof}[Proof of Theorem \ref{extremalquadric}] The argument follows the lines of the proof of \cite[Theorem 1]{PV}, see also \cite{W} for a one-dimensional version.
By the Borel-Cantelli lemma, it is enough to check that for all $\eps>0$,
\[ \sum_{k\geq 1} \mu\left(\big\{x\in X\ \big|\ \exists\, v\in X(\Q):\, \left\{
\begin{array}{c}
2^k\leq H(v) < 2^{k+1}\\
\dist(x,v)\leq 2^{-k(1+\eps)}
\end{array}\right.\big\}\right)
< \infty.\]
Fix $k\geq 1$.
There exists an integer $K$ such that we may cover $X$ by a family of balls $B_i=B(x_i,2^{-k(1+\frac{\eps}{2})})$, $i=1,\dots,N$, so that any intersection of more than $K$ distinct balls is empty.
By Lemma~\ref{simplexquadric}, for $k$ large enough, for each $i$, the set of points $v\in X(\Q)\cap B_i$ satisfying $2^k\leq H(v)<2^{k+1}$ is contained in a totally isotropic rational subspace $L_i$, and therefore, by the IAD property of $\mu$ for some $C,\alpha > 0$ one has
\begin{align*}
\mu\left(\big\{x\in B_i\ \big|\ \exists\, v\in X(\Q):\, \left\{
\begin{array}{c}
2^k\leq H(v) < 2^{k+1}\\
\dist(x,v)\leq 2^{-k(1+\eps)}
\end{array}\right.\big\}\right)
& \leq \mu\left(B_i\cap L_i^{(2^{-k(1+\eps)})}\right)\\
& \leq C2^{-k\alpha\frac{\eps}{2}}\mu(B_i).
\end{align*}
Summing over all balls $B_i$, and using the fact that the cover $(B_i)_{i\in\N}$ has multiplicity at most $K$, we get
\[ \mu\left(\big\{x\in X\ \big|\ \exists\, v\in X(\Q):\, \left\{
\begin{array}{c}
2^k\leq H(v) < 2^{k+1}\\
\dist(x,v)\leq 2^{-k(1+\eps)}
\end{array}\right.\big\}\right)
\leq KC2^{-k\alpha\frac{\eps}{2}}.\]
This finishes the proof of the theorem.
\end{proof}

\begin{remark} 
When $\rk_\Q(X) = 1$, all the subspaces $L$ appearing in Definition \ref{def-ad} are zero-dimensional, and isotropic absolute decay coincides with weak absolute decay as defined in \cite{bgsv}. 
Moreover, in the case where $X$ is a sphere, Theorem \ref{extremalquadric} can be viewed as a corollary of \cite[Theorem~2]{bgsv}.
\end{remark}
\begin{remark}
We could have stated a slightly stronger version of the theorem, in the form of a Khintchine-type theorem: if $\mu$ is $\alpha$-IAD, 
%with associated parameters $C,\alpha$, 
and if $\psi:\R^+\to\R^+$ is a non-increasing function satisfying $$\sum_{k\in\N}k^{\alpha-1}\psi(k)^\alpha<\infty,$$ then for $\mu$-almost every $x$ in $X$  there exists $c>0$ such that
\[ \forall v\in X(\Q),\ \dist(x,v)\geq c\psi\big(H(v)\big).\]
With some minor modifications, our proof works in this slightly more general setting.
\end{remark}

\subsection{Hausdorff dimension and Diophantine exponents}\label{sec:diophantine2}

As a complement to the above study of the extremality problem, we explain here how the simplex lemma can be used to give a simple proof of a recent result of Fishman, Merrill and Simmons \cite{fishmanmerrillsimmons}.
Once again, $X$ denotes a rational quadric projective hypersurface of dimension $n$.
Given $\beta>0$, we shall be concerned with the set
\[ W_\beta=\{x\in X \ |\ \beta(x)\geq\beta\}.\]
%Recall from Theorem~\ref{extremalquadric} that whenever $\beta>1$, the set $W_\beta$ has measure zero; in fact, it is proved in \cite[Theorem~6.4]{fkmsquadric} that $\dim_H W_\beta=\frac{n}{\beta}$.
Given a subset $K$ in $X$, our goal will be to bound the Hausdorff dimension of the intersection $K\cap W_\beta$; we shall be able to do so if $K$ is the support of a sufficiently regular measure.

For $\delta>0$, a Borel measure $\mu$ on a metric space $X$ is said to be \emph{Ahlfors-regular} of dimension $\delta$ if we have, for some constant $A>0$,
\[ \forall\, x\in X\,\forall r\,\in(0,1],\quad
\frac{1}{A}r^\delta
\leq \mu(B(x,r))
\leq Ar^\delta.\]
%Recall also that we say that $\mu$ is $\alpha$-isotropically decaying
%if it satisfies, for some $C\geq 0$, for every $x\in X$ and every totally isotropic rational subspace $L\leq X$,
%\[ \forall\eps,\rho>0,\ 
%\mu(B(x,\rho)\cap L^{(\eps\rho)}) \leq C\eps^\alpha\mu(B(x,\rho)),\]
%where $L^{(\tau)}$ denotes the neighborhood of size $\tau$ of the subspace $L$.
%The measure $\mu$ is said to be \emph{doubling} if there exists a constant $C\geq 0$ such that for all $x\in X$ and $r\in(0,1)$, $\mu(B(x,2r))\leq C\mu(B(x,r))$.
We now present a short proof of a strengthening of \cite[Theorem~1.2]{fishmanmerrillsimmons}, using Lemma \ref{simplexquadric}.
%the version of the simplex lemma given in \S\ref{sec:simplex}.

\begin{theorem}\label{hausdorffexponent}
Let $X$ be a rational quadric projective hypersurface.
Let $\mu$ be an Ahlfors-regular  measure of dimension $\delta$ on $X$, and let $K=\Supp\mu$.
If $\mu$ is $\alpha$-IAD, then we have, for all $\beta\geq 1$,
\eq{dimbound}{\dim_H (K\cap W_\beta) \leq \delta-\alpha(1-\frac{1}{\beta}).}
\end{theorem}

\begin{remark} \cite[Theorem~1.2]{fishmanmerrillsimmons} establishes \equ{dimbound} under a stronger assumption  that $\mu$ is $\alpha$-absolutely decaying. 
%The statement we have here is in fact stronger than the one in \cite{fishmanmerrillsimmons}, because
  However in our decay condition we only have to consider totally isotropic subspaces.
%This more intrinsic approach to the problem is also what makes the solution simpler.
%\end{remark}
%
%\begin{remark}
%Note that 
In particular, Theorem \ref{hausdorffexponent} covers the case where $K$ is a smooth submanifold of $X$ of dimension at least $\rk_Q(X)$, and therefore generalizes Theorem \ref{bei}.
\end{remark}

The proof of Theorem~\ref{hausdorffexponent}  is a straightforward adaptation of that of \cite[Theorem 2]{PV}. We shall use the easy Hausdorff--Cantelli lemma stated below.

\begin{lemma}[Hausdorff--Cantelli]
Let $(B_i)_{i\geq 0}$ be a family of balls in %$\R^d$,
a metric space, and assume that $\sum_{i\geq 0}(\diam B_i)^s<\infty$.
Then,
\[ \dim_H (\limsup B_i) \leq s.\]
\end{lemma}
\begin{proof}
Left as an exercise, see Bernik--Dodson \cite[Lemma~3.10]{bernikdodson}.
\end{proof}

\begin{proof}[Proof of Theorem~\ref{hausdorffexponent}]
If $\beta=1$, there is nothing to prove, so we assume $\beta>1$ and fix $\gamma\in(1,\beta)$.
For $p\geq 0$, let
\[ A_p = \big\{x\in X\ \big|\ \exists \,v\in X(\Q):\, \left\{
\begin{array}{c}
2^p\leq H(v) < 2^{p+1}\\
\dist(x,v)\leq 2^{-\gamma p}
\end{array}\right.\big\}
.\]
Taking a maximal $2^{-p}$-separated subset $\{x_i\}_{1\leq i\leq \ell_p}$ of $K\cap A_p$, the collection of balls $\cC_p=\big(B(x_i,2^{-p})\big)_{1\leq i\leq \ell_p}$ covers $K\cap A_p$ and has multiplicity bounded above by some constant $C$ depending only on $X$.
Using the Ahlfors regularity of $\mu$, this implies $\ell_p2^{-p\delta}\leq AC\mu(X)=AC$, i.e. $\ell_p\leq AC 2^{p\delta}$.

Since $\gamma>1$, Lemma \ref{simplexquadric} shows that for $p$ large enough, for each ball $B\in\cC_p$, there exists a totally isotropic subspace $L_B$ of $X$ such that $A_p\cap B\subset L_B^{(2^{-\gamma p})}$.
So the decay condition on $\mu$ yields, within multiplicative constants depending only on $X$ and $\mu$, that
\[ \mu(A_p\cap B) \ll 2^{-(\gamma-1)\alpha p}\mu(B) \asymp 2^{-p[\delta+(\gamma-1)\alpha]}.\]
Next, take a minimal cover $\cD_{B}=(B_i)_{i\in I_B}$ of the set $K\cap A_p\cap B$ by balls of radius $2^{-\gamma p}$ centered on $K\cap A_p\cap B$.
Just as above, the Ahlfors regularity of $\mu$ shows that $\# I_B\ll 2^{\delta\gamma p}\mu(A_p\cap B) \ll 2^{p\gamma \delta} 2^{-p[\delta+(\gamma-1)\alpha]}$.
%for the first estimate, $\card I_B \ll 2^{\delta\gamma p}\mu(2^{-\gamma p}-neighborhood of A_p\cap B)$ and it is not hard to check that the measure of the $2^{-\gamma p}$-neighborhood of $A_p\cap B$ satisfies the same measure estimate as $A_p\cap B$ (same proof, using the simplex lemma).
Thus, we find for every $s>0$,
\[ \sum_{B\in\cC_p}\sum_{i\in I_B} (\diam B_i)^s \ll 2^{p\delta}2^{p(\gamma-1)(\delta-\alpha)}2^{-p\gamma s} = 2^{-p[s\gamma-\gamma\delta+\alpha(\gamma-1)]}\]
If $s > \delta-\alpha(1-\frac{1}{\gamma})$, then the family of balls $(B_i)_{i\in I_B,\, B\in\cC_p,p\in\N}$ satisfies the assumption of the Hausdorff--Cantelli lemma, and therefore, letting 
\[ s\to\delta-\alpha(1-{1}/{\gamma})\]
we find that $\dim_H(\limsup B_i)\leq\delta-\alpha(1-\frac{1}{\gamma})$.
Now, since $\gamma<\beta$, we have $K\cap W_\beta\subset(\limsup B_i)$, hence letting $\gamma\to\beta$, we can conclude that  the \hd\ of $K\cap W_\beta$ is not greater than $\delta-\alpha(1-\frac{1}{\beta})$.
\end{proof}

In the case of $\Q$-rank one, any Ahlfors-regular measure of dimension $\delta$ is $\delta$-IAD, so we get the following corollary, which applies in particular when $X=\NS^{n-1}$ is the unit sphere in $\R^{n}$:

\begin{corollary}
Let $X$ be a rational quadric hypersurface of $\Q$-rank one, and let
 $\mu$ be an Ahlfors-regular measure of dimension $\delta$ on $X$.
Writing $K=\Supp\mu$, we have, for every $\beta\geq 1$,
$\dim_H (K\cap W_\beta) \leq \frac{\delta}{\beta}$.
\end{corollary}

\subsection{Badly approximable points}\label{sec:diophantine3}

%Dirichlet's principle for the quadric hypersurface $X$, as exposed in \cite{fkmsquadric}, states that for every $x\in X$, there exists a constant $C_x$ such that the inequality $d(x,v)\leq C_xH(v)^{-1}$ has infinitely many solutions $v\in X(\Q)$.
%So is it natural to say that a point $x\in X$ is \emph{badly approximable} if there exists $c>0$ such that
%\[ \forall v\in X(\Q),\ d(x,v)\geq cH(v)^{-1}.\]
Recall the definition \equ{defba} of the set $\BA_X$ of intrinsically badly approximable points in $X$.
%It is shown in \cite[Theorem~4.5]{fkmsquadric}
%that the set $\BA_X$ of badly approximable points in the quadric hypersurface $X$ has full Haudsorff dimension.
%In fact, we know from \cite[Theorem~4.3]{fkmsgeneral} that 
As was mentioned in Section \ref{exposition}, it is known \cite{fkmsgeneral} to satisfy some winning properties in the sense of Schmidt's games.
Our goal will now be to give a more elementary proof of a refinement of the winning property, again using the simplex lemma.

In order to study the properties of badly approximable numbers, Schmidt introduced in \cite{schmidt_games} a certain family of games, and the associated {winning} property.
Those games were subsequently studied in numerous papers, among which \cite{bfkrw} is the most relevant for the present purposes.

We now explain the principles of our version of Schmidt's game.
As before, $X$ is a rational quadric hypersurface of $\PP^{n}(\R)$.
There are two players, Alice  and Bob, and some parameter $\beta\in(0,\frac{1}{3})$.
To start, Bob chooses a ball $B_0=B(x_0,\rho_0)$ in $X$.
Then, at each stage of the game, after Bob has chosen a ball $B_i=B(x_i,\rho_i)$, Alice chooses a totally isotropic rational subspace $L$ of $X$ and deletes its neighborhood of size $\eps$, with $0<\eps\leq\beta\rho_i$.

A set $S$ is {\sl isotropically $\beta$-winning} if 	Alice can make sure that $$\bigcap B_i\cap S\neq\varnothing.$$
Finally, $S$ is  {\sl isotropically winning} if it is isotropically $\beta$-winning for arbitrarily small $\beta>0$. This is a strengthening of the hyperplane absolute winning property as defined in \cite{bfkrw} (for the latter, Alice is allowed to delete neighborhoods of arbitrary hyperplanes).
Thus  the following theorem is a refinement of \cite[Theorem~4.3]{fkmsgeneral}:

\begin{theorem}[Badly approximable points on $X$ are winning]
\label{winning}
Let $X$ be a rational quadric hypersurface in $\PP^{n}(\R)$.
Then the set $\BA_X$ %of badly approximable points in $X$ 
is isotropically winning.
\end{theorem}
\begin{proof}
Fix $\beta\in(0,\frac{1}{3})$.
Bob first picks a ball $B_0=B(x_0,\rho_0)$.
By Lemma~\ref{simplexquadric}, there exists a constant $c>0$ depending only on $X$ such that all rational points $v$ in $2B_0$ satisfying $H(v)\leq c\rho_0^{-1}$ are included in some totally isotropic rational subspace $L_0$.
Alice deletes $L_0^{(\beta\rho_0)}$.
Similarly, once Bob has chosen a ball $B_i=B(x_i,\rho_i)$, the rational points $v\in 2B_i$ such that $H(v)\leq c\rho_i^{-1}$ all lie on a hyperplane $L_i$, and Alice deletes $L_i^{(\beta\rho_i)}$.
If there is no rational point of small height in $B_i$, then Alice can delete a ball of radius $\beta\rho_i$ around the center.
This ensures that $\rho_i\to 0$.

We claim that this strategy forces $\bigcap_{i\geq 0} B_i\subset\BA_X$.
To see this, let $x\in\bigcap B_i$  and $v\in X(\Q)$.
Choose $i$ such that
\begin{equation}\label{rhoi}
c\rho_{i-1}^{-1}
\leq H(v) \leq c\rho_{i}^{-1}.
\end{equation}
If $v\not\in 2B_i$, then, using $x\in B_i$, we find
\[ \dist(x,v)
\geq \rho_i
\geq \beta\rho_{i-1}
\geq \beta cH(v)^{-1}.\]
And if $v\in 2B_i$, then \eqref{rhoi} implies that $v\in L_i$, and since $x\in B_{i+1}$,
\[ \dist(x,v)
\geq \beta\rho_i
\geq \beta^2\rho_{i-1}
\geq \beta^2 cH(v)^{-1}.\]
Taking $c_0=c\beta^2$, we find
\[ \forall\, v\in X(\Q),\ \dist(x,v)\geq c_0H(v)^{-1},\]
so $x\in\BA_X$.
\end{proof}

As is the case with the hyperplane absolute game, the advantage of the isotropic game is the inheritance of winning properties to sufficiently regular subsets.
More precisely, given a compact subset $K\subset X$, we may consider the isotropic game played on $K$.
The rules are the same as before, but the ambient metric space is now $K$: at each stage, Bob chooses a ball $B(x_i,\rho_i)$ centered on $K$, and Alice deletes the intersection of $K$ with the neighborhood of size $\beta\rho_i$ of a rational isotropic subspace.
Naturally, we shall say that a set $S$ is {\sl isotropically winning on $K$} if $S\cap K$ is winning for the isotropic game on $K$.

Following Broderick--Fishman--Kleinbock--Reich--Weiss \cite{bfkrw}, let us
say that a subset $K\subset X$ is {\sl isotropically diffuse} if there exists $\beta,\rho_K>0$ such that for every $\rho\in(0,\rho_K)$, $x\in K$, and every totally isotropic subspace $L$, the set $$K\cap B(x,\rho)\smallsetminus L^{(\beta\rho)}$$ is non-empty.
This is a quantitative way to say that $K$ is nowhere included in a small neighborhood of a totally isotropic subspace.
The next lemma is a straightforward analogue of \cite[Proposition~4.9]{bfkrw}:

\begin{lemma}
%[Proposition~4.9 in \cite{bfkrw}]
\label{diffuse}
%\textnormal{\cite[Proposition~4.9]{bfkrw}}\textbf{.}
Let $X$ be a rational quadric hypersurface in $\PP^{n}(X)$. %\comm{Do we need to say that it has rank $k\geq 1$?}  
%$\Q$-rank or $\R$-rank?
If $L\subset K$ are two isotropically diffuse subsets of $X$, and $S\subset X$ is isotropically winning on $K$, then $S$ is isotropically winning on $L$.
\end{lemma}
%Schmidt's winning property is a way to express that a set is large.
%For example, any winning set is dense and has maximal Hausdorff dimension.
%In fact, one can show the following stronger corollary of Theorem~\ref{winning}.
The proof   is identical to the argument presented in \cite{bfkrw}, with obvious modifications to our setting;  we refer the reader to \cite[Section~4]{bfkrw} for details.

\smallskip
It follows from the above lemma and Theorem \ref{winning} that $\BA_X$ is isotropically winning on any isotropically diffuse subset of $X$. This in particular applies to smooth submanifolds $M$ of $X$ of dimension not less than $\rk_\Q(X)$, which are clearly isotropically diffuse. Furthermore, the Lebesgue measure on $M$ as above is clearly Ahlfors-regular of dimension equal to $\dim M$. Therefore, in view of \cite[Lemma~5.3]{bfkrw}, 
for every open subset $U$ of $X$ such that $U\cap M\neq\varnothing$, one has 
\[ \dim_H (M\cap \BA_X\cap U) = \dim M,\]
which implies Theorem \ref{bai}.

%\begin{corollary}
%\label{ba_sm}
%Let $X$ be a rational quadric hypersurface in $\PP^{n}$, with $\Q$-rank $k\geq 1$.
%Let $M$ be a connected $C^1$ submanifold in $X$, with $\dim M\geq k$.
%Then $\BA_X$ is isotropically winning on $M$.
%In particular $\BA_X\cap M$ has maximal Hausdorff dimension in $M$.
%\end{corollary}

\begin{remark}
In the case of $X = \NS^{n-1}$, or more generally of a rational quadric of $\Q$-rank one, the above corollary shows that $\BA_{X}$ is winning on any positive-dimensional submanifold of $X$.
%Note that the problem of showing that $\BA_{\NS^n}$ is winning on any (non-degenerate) submanifold on the sphere is \emph{much} easier than on the projective space,
This can be compared with a similar question for \da\ in Euclidean spaces, 
for which it is still open, despite recent progress due to Beresnevich \cite{beresnevich_ba} and Yang \cite{yang_ba}.
\end{remark}

\section{Further directions and open problems}
\label{sec:open}

\textbf{Khintchine's theorem.}
It would be interesting to use the geometric observations of this note to give an elementary proof of Khintchine's theorem on quadric hypersurfaces, due to Fishman, Kleinbock, Merrill and Simmons \cite[Theorem~6.3]{fkmsquadric}.

\bigskip

\noindent\textbf{Singular points.}
Given a rational quadric $X$ in $\PP^{n}(X)$, one may define, for $c>0$,
\[ D(c) = \left\{x\in X \left|\begin{aligned} 
\exists\, N_0:\,\forall\, N\geq N_0\ \exists\, v\in X(\Q) \text{ such that}\\
H(v) \leq N\text{  and }\dist(x,v) \le \frac c{\sqrt{NH(v)}}\end{aligned}\right.\right\}, \]
and call a point $x\in X$ \emph{singular} if $x\in\bigcap_{c>0}D(c)$.
%When $X$ is a unit sphere, it is proved in  \cite[Theorem 4.1]{kleinbockmerrill} that $D(C) = X$ for $C$ large enough; an explicit estimate for such $C$ was later found by Moshchevitin \cite{moshchevitin}{\tt https://arxiv.org/pdf/1510.06214.pdf}. The same holds for any quadric of $\R$-rank $1$ \cite[Theorem~5.1(iii)]{fkmsquadric}.
If $X$ has $\Q$-rank $1$, it follows from Dani's work \cite{dani_divergent} that $x$ is singular if and only if $x\in X(\Q)$.
%Indeed, it is readily seen from Proposition~\ref{daniquadric} that if $x$ is singular, then the orbit $g_t^x\Z^d$ is divergent in the space $G(\R)/G(\Z)$.
%When this space has rank $1$, Dani has showed that every divergent diagonal trajectory comes from a rational point.
%One question remains open: can one find a non-singular point $x$ such that the trajectory $g_t^x\Z^d$ is divergent? Note that this is impossible when $X$ has $\Q$-rank one, because then the divergent trajectories are all rational.
In fact, one can show that if $X$ has $\Q$-rank $1$, $D(c)=X(\Q)$ for $c>0$ small enough.
This follows for example from the following strengthening of Lemma~\ref{simplexquadric}, whose proof is identical up to some minor changes. See also \cite[Theorem~3]{sumofsquares} for an alternative proof.
%\comm{Please add a reference to  {\tt https://arxiv.org/pdf/1803.09306.pdf}.}

\begin{lemma}[A stronger simplex lemma for quadric hypersurfaces]
Let $X$ be a rational quadric hypersurface in $\PP^{n}(\R)$.
Then there exists $c>0$ such that, for every $x\in X$ and any $\rho \in(0,1)$, 
the set
\[  \left\{v\in X(\Q)\ \left|\ H(v)\leq c\rho^{-1}, \ \dist(x,v) \le \sqrt{\frac{\rho}{H(v)}}\right.\right\} \]
is contained in a totally isotropic rational subspace $L\subset X$.
\end{lemma}
When the quadric $X$ has $\Q$-rank at least $2$, it is natural to expect that there exist some nontrivial singular points. It might then be interesting to compute the Hausdorff dimension of the set of singular points on $X$, similarly to what has been done in \cite{cheung,cheung-chevallier} for \da\ in the Euclidean space.

\bigskip

\noindent\textbf{Extremality.}
In view of the definitive results in the area of \da\ on manifolds and fractals obtained in \cite{KLW, kleinbock-margulis}, it is %Theorem~\ref{extremalquadric} and Corollary~\ref{rankoneextremal}, it is 
natural to attempt to weaken the condition of isotropic absolute decay of $\mu$ as in Theorem~\ref{extremalquadric}, and
conjecture that on a general quadric hypersurface, any analytic submanifold that is not included in an isotropic  subspace is extremal.
In fact, by analogy with \cite{kleinbock_inheritance, kleinbock_inheritance2}, one can guess that an analytic submanifold on a quadric hypersurface inherits its Diophantine exponent from the smallest totally isotropic subspace containing it.
The proof of these two facts requires an appropriate quantitative non-divergence statement in the space of lattices, and should appear elsewhere.

\bigskip

\noindent\textbf{Other projective varieties.}
One may wonder how general is the approach presented here, and whether it can be used to study intrinsic Diophantine approximation on varieties that are not quadric hypersurfaces.
Some partial answers are given in \cite{breuillardsaxce_subspace}, where Diophantine approximation on a generalized flag variety is studied.
However, many questions remain open in this generality, for which we refer the reader to \cite{breuillardsaxce_subspace}.

\Addresses


\begin{thebibliography}{15}
   
   
   \bibitem{beresnevich_ba}
   V.\ Beresnevich,
  {\em Badly approximable points on manifolds},
   Invent. Math. {\bf 202}, no.\ 3 (2015), 1199--1240.

  \bibitem{bgsv}
   V.\ Beresnevich, A.\ Ghosh, D.\ Simmons and S.\ Velani,
  {\em Diophantine approximation in {K}leinian groups: singular, extremal and bad limit points} (2016), 
preprint available at  {\tt    arXiv:1610.05964}. 

   \bibitem{bernikdodson}
   V.\ Bernik and M.\ Dodson,
  {\em Metric {D}iophantine approximation on manifolds},
  Cambridge Tracts in Mathematics,
vol.\ 137,
Cambridge University Press, Cambridge, 1999.

 \bibitem{breuillardsaxce_subspace} E.\ Breuillard and N.\ de Saxc\'e,  {\em A subspace theorem for manifolds}, in
preparation.

 \bibitem{bfkrw}  R.\ Broderick, L.\ Fishman, D.\ Kleinbock, A.\ Reich  and B.\ Weiss,  {\em The set
of badly approximable vectors is strongly $C^1$
incompressible},  Math.\ Proc.\ Cambridge Philos.\ Soc.\ {\bf 153}, no.\ 2 (2012), 319--339.

 \bibitem{cassels}  J.\,W.\,S.\ Cassels,  {\em An introduction to Diophantine approximation},   Cambridge Tracts in Mathematics and Mathematical Physics, vol.\ 145, Cambridge University Press, New York, 1957.

 \bibitem{cheung}  Y. Cheung,  {\em  Hausdorff dimension of the set of singular pairs},   Ann.\  Math.\ 
{\bf 173}, no.\ 1 (2011), 127--167.

 \bibitem{cheung-chevallier}  Y.\ Cheung and N.\ Chevallier,  {\em  Hausdorff dimension of singular vectors},  
Duke Math.\ J.\  {\bf 165}, no.\ 12 (2016),  2273--2329.

 \bibitem{dani_divergent}  S.\, G.\ Dani,  {\em  Divergent trajectories of flows on homogeneous spaces and
Diophantine approximation},   J.\ Reine Angew.\ Math. {\bf 359} (1985), 55--89.

 \bibitem{davenport}  H.\ Davenport,  {\em  A note on Diophantine approximation. II},   Mathematika 
{\bf 11} (1964), 50--58.

 \bibitem{dickinsondodson}  H.\ Dickinson and M.\,M.\ Dodson,  {\em  Simultaneous Diophantine approximation
on the circle and Hausdorff dimension},   Math.\ Proc.\ Cambridge Philos.\ Soc.\ {\bf 130}, no.\ 3 (2001), 515--522.

 \bibitem{drutu}  C. Dru\c tu,  {\em  Diophantine approximation on rational quadrics},   Math.\ Ann.\ {\bf 333}, no.\ 2 (2005), 405--469.


   \bibitem{fkmsquadric}
 L.\ Fishman,
D.\ Kleinbock, K.\ Merrill anf  D.\ Simmons, 
{\em Intrinsic Diophantine approximation on quadric hypersurfaces} (2014),
preprint available at  {\tt arXiv:1405.7650}.

    \bibitem{fkmsgeneral}
L.\ Fishman,
D.\ Kleinbock, K.\ Merrill anf  D.\ Simmons, 
{\em Intrinsic Diophantine approximation on manifolds: general theory},  Trans.\ Amer.\ Math.\ Soc.\ {\bf 370}. no.\ 1 (2018), 577--599.

 \bibitem{fishmanmerrillsimmons} L. Fishman, K.\ Merrill  and D.\ Simmons,  {\em  Hausdorff dimensions of very
well intrinsically approximable subsets of quadratic hypersurfaces} (2015),   preprint available at  {\tt arXiv:1502.07648v2}.

 \bibitem{kleinbock_inheritance} D. Kleinbock,  {\em  Extremal subspaces and their submanifolds},   Geom.\ Funct.\
Anal.\  {\bf 13}, no.\ 2 (2003), 437--466.

 \bibitem{kleinbock_inheritance2} D. Kleinbock,  {\em  An extension of quantitative nondivergence and applications to Diophantine exponents},   Trans.\ Amer.\ Math.\ Soc.\ {\bf 360} (2008), 6497--6523.


 \bibitem{KLW} D.\ Kleinbock, E.\ Lindenstrauss  and B.\ Weiss,  {\em  On fractal measures and
Diophantine approximation},   Selecta Math.\ (N.S.) {\bf 10}, no.\ 4 (2004), 479--523.

 \bibitem{kleinbock-margulis} D.\ Kleinbock and G.\,A.\ Margulis,  {\em Flows on homogeneous spaces and Diophantine approximation on manifolds},   Ann.\ Math.\ {\bf 148}, no.\ 1 (1998), 339–360.

   \bibitem{kleinbockmerrill}
D.\ Kleinbock and  K.\ Merrill,
{\em Rational approximation on spheres},
Isr.\ J.\ Math.\ {\bf 209} (2015), no.\ 1, 293--322. 


\bibitem{sumofsquares} D.\ Kleinbock and N.\ Moshchevitin,  {\em  Simultaneous Diophantine approximation: sums of squares and homogeneous polynomials},  preprint available at  {\tt    arXiv:1803.09306v2}, Acta Arithmetica, to appear.

\bibitem{kristensenthornvelani} S.\ Kristensen, R.\ Thorn  and S.\ Velani, {\em  Diophantine approximation and
badly approximable sets},   Adv.\ Math.\  {\bf 203}, no.\ 1 (2006), 132--169.

\bibitem{PV} A.\ Pollington and S.\ Velani,  {\em  Metric diophantine approximation and absolutely friendly measures},   Selecta Math.\ (N.S.), {\bf 11}, no.\ 2 (2005), 297--307.

\bibitem{schmidt_games} W.\,M.\ Schmidt, {\em  On badly approximable numbers and certain games},   Trans.\
Amer.\ Math.\ Soc.\ {\bf 123} (1966),  178--199.

\bibitem{schmidt} W. M. Schmidt, {\em Diophantine approximation}, Lecture Notes
in Mathematics,  vol.\ 785, Springer, Berlin, 1980.

\bibitem{W} B. Weiss, {\em  Almost no points on a Cantor set are very well approximable},   Proc.\ R.\ Soc.\ Lond.\ A  {\bf 457} (2001), 949--952.

\bibitem{yang_ba} L.\ Yang, {\em Badly approximable points on curves and unipotent orbits in
homogeneous spaces} (2017), preprint available at  {\tt    arXiv:1703.03461v1}.


 \end{thebibliography}
\end{document}